\documentclass[a4paper,12pt,reqno]{amsart}

\usepackage{amsthm}
\usepackage{amsmath}
\usepackage{amssymb}
\usepackage{mathrsfs}
\usepackage{latexsym}
\usepackage{exscale}
\usepackage{geometry} 
\usepackage{graphicx} 
\usepackage[utf8]{inputenc}
\usepackage{verbatim}
\usepackage{enumerate} 

\usepackage{fancyhdr}

\usepackage{tikz}
\usepackage{caption}

\usepackage{color}
\definecolor{red}{rgb}{1,0.1,0.1}
\definecolor{blue}{rgb}{0.1,0.1,1}
\definecolor{vb}{RGB}{160,32,240}

\theoremstyle{plain}

\newtheorem*{teo*}{Theorem}
\newtheorem*{prop*}{Proposition}
\newtheorem*{lema*}{Lemma}

\numberwithin{equation}{section}
\newtheorem{teo}{Theorem}[section]

\newtheorem{prop}[teo]{Proposition}

\theoremstyle{remark}
\newtheorem{remark}[teo]{Remark}

\theoremstyle{definition}

\newtheorem*{mydef*}{Definition}

\calclayout

\allowdisplaybreaks

\newcommand{\R}{\mathbb{R}^n}
\newcommand{\Z}{\mathbb{Z}}
\begin{document}

\baselineskip=17pt

\title[Weights for  maximal  operators]{Weighted boundedness for the maximal operator associated with matrices}

\author[Gonzalo~Iba\~{n}ez-Firnkorn]{Gonzalo Iba\~{n}ez-Firnkorn}
\address{(Gonzalo~Iba\~{n}ez-Firnkorn) 
Instituto de Matemática (INMABB), Departamento de Matemática, Universidad Nacional del Sur (UNS)-CONICET
\\Bahía Blanca, Argentina}
\email{gonzalo.ibanez@uns.edu.ar}


\begin{abstract}
In this paper we study the boundedness  on $L^p(w)$ of the maximal ope\-ra\-tor $M_{A^{-1}}$, defined by $M_{A^{-1}}f(x)=Mf(A^{-1}x)$, that is, the maximal of Hardy-Littlewood composed with an invertible matrix $A$. 
We present a new approach of the boundedness for general matrices, and a characterization for finite order matrices.
Also, we show examples illustrating the difference between the class of weights with a matrix, $\mathcal{A}_{A,p}$, and the classical Muckenhoupt weight class, $\mathcal{A}_{p}$.

Finally, we extend these results to the fractional framework, considering the fractional maximal operator $M_{\alpha, A^{-1}}$.
\end{abstract}

\thanks{The author was partially supported by CONICET and SGCyT-UNS, 2023-2026, Res. CSU-461/2023, código 24/L126.
}



\maketitle

\section{Introduction}

Let $(\mathbb{R}^n,\mu)$ be a measurable space and $A\in \mathbb{R}^{n\times n}$ be an invertible matrix.  We define a maximal operator associated to the matrix $A$ and the measure $\mu$ as follows. For any  locally integrable function $f$, we set 
\[M^{\mu}_{A^{-1}}f(x)=M^{\mu}f(A^{-1}x)
=\sup_{Q\ni A^{-1}x}\frac1{\mu(Q)}\int_{Q}|f(y)|\,d\mu(y)
\]
where the supremum is taken over all cubes $Q$ with side parallel to the axes. 
If we take $\mu$ as the Lebesgue measure $dx$, we denote $M_{A^{-1}}=M^{dx}_{A^{-1}}$.
The fractional maximal operator  associated to  $A$ and $\mu$ is defined as
\[M^{\mu}_{\alpha,A^{-1}}f(x)=M^{\mu}_{\alpha}f(A^{-1}x)
=\sup_{Q\ni A^{-1}x}\frac1{\mu(Q)^{1-\frac{\alpha}{n}}}\int_{Q}|f(y)|\,d\mu(y)
\]
where the supremum is taken over all cubes $Q$ with side parallel to the axes. 
If we take $\mu$ as the Lebesgue measure $dx$, we denote $M_{\alpha,A^{-1}}=M^{dx}_{\alpha,A^{-1}}$.

This type of maximal operator appears naturally on the boundedness of integral operator with several singular points. These integral operators are defined by 
\begin{equation}\label{defT}
Tf(x)=\int_{\mathbb{R}^n} k_1(x-A_1y)k_2(x-A_2y)f(y)\,dy,
\end{equation}
where the function $k_i$ satisfy certain size and regularity conditions, $A_1$ and $A_2$ are invertible matrices such that $A_1-A_2$ are invertible and $f\in L^{\infty}_{\text{loc}}(\mathbb{R}^n)$. This operator generalizes to $m$ singular points using $m$ functions $k_i$.

 In  \cite{RS88} appears the first integral operator of this type, in this case the kernel of the operator is $K(x,y)=|x-y|^{-\alpha}|x+y|^{\alpha-1}$, with $0<\alpha<1$, in the paper the authors studied the boundedness on $L^2(\mathbb{R})$. 

In the case of $k_i$ are fractional rough kernel, that is $k_i(z)=|z|^{-\alpha_i}$, with $\alpha_1+\alpha_2=n-\alpha$ and $0<\alpha<n$, in \cite{RiU13} the authors proved a Coifman-Fefferman inequality 
$$\int_{\mathbb{R}^{n}} |Tf(x)|^qw(x)^q \,dx \lesssim \sum_{i=1}^2 \int_{\mathbb{R}^{n}} |M_{\alpha,A^{-1}_i}f(x)|^qw(x)^q \,dx$$
for all $0<q<\infty$, and $w^q$ a weight in the $\mathcal{A}_{\infty}$ Muckenhoupt class.

When $k_i$ satisfy general Hörmander and size conditions, the Coifman-Fefferman inequalities were established in \cite{IFR18,IFR20}.
Initially, in \cite{RiU13, RiU14, IFR18, IFR20} the authors consider weights in Muckenhoupt class, $\mathcal{A}_p$, such that $w(Ax)\lesssim w(x)$ almost everywhere $x\in \mathbb{R}^n$. The power weights provide a typical example of functions that satisfy these conditions for any invertible matrix $A$.

The boundedness of integral operators of the form \eqref{defT} were studied in several context: Hardy spaces \cite{RoU12,RoU14}; variable Lebesgue spaces \cite{RoU13,UV20,UV20a}; another operator are studied in \cite{GU93,RiU13,RiU14}; commutators in \cite{IFR20}.

An important question is to determine the weights for which the maximal operator is bounded on weighted Lebesgue spaces.

Recently in \cite{IFRV22}, the authors introduced the weight class $\mathcal{A}_{A,p}$. A weight $w$ belongs in the $\mathcal{A}_{A,p}$ class, with $1<p<\infty$ , if
$$
[w]_{\mathcal{A}_{A,p}}=\sup_Q \frac1{|Q|}\int_Q w_A(x)\,dx \left(\frac1{|Q|}\int_Q w(x)^{-\frac1{p-1}}\,dx\right)^{p-1}
<\infty
$$
where $w_A(x)=w(Ax)$ and the supremum is taken over all cubes $Q$ with side parallel to the axes.

Also, $w$ belongs in the $\mathcal{A}_{A,1}$ if 
$$
[w]_{\mathcal{A}_{A,1}}=\underset{x \in \mathbb{R}^n}{{\rm ess}\sup} \frac{M_{A^{-1}}w(x)}{w(x)} <\infty. 
$$

They proved that $M_{A^{-1}}$ is bounded from $L^p(w)$ into $L^{p,\infty}(w)$ if and only if $w\in \mathcal{A}_{A,p}$. Also, proved that  $M_{A^{-1}}$ is bounded on $L^p(w)$ if and only if $w$ satisfies an appropriate testing condition.
 In this paper, we provide more tractable conditions for the strong-type boundedness.

Now, we recall some properties of the class $\mathcal{A}_{A,p}$ studied in \cite{IFRV22},
\begin{itemize}
\item If $w\in \mathcal{A}_p$ and $w(Ax)\lesssim w(x)$ a.e. $x\in \mathbb{R}^n$ then $w\in \mathcal{A}_{A,p}$.
\item If $w\in \mathcal{A}_{A,p}$ then $w(Ax)\leq [w]_{\mathcal{A}_{A,p}} w(x)$.
\end{itemize}

A natural question is to determine which weights guarantee the boundedness of $M_{A^{-1}}$ on $L^p(w)$. In this direction, we provide sufficient conditions for general matrices and a complete characterization for the finite-order case.

For the state of our first main result,  we define the bump version of the $\mathcal{A}_{A,p}$ as follows. Let $\varphi$ be a Young function,  $w$  belongs in the class $ \mathcal{A}_{A,p,\varphi}$ if
$$
[w]_{\mathcal{A}_{A,p,\varphi}}=\sup_Q \|w_A^{\frac1{p}}\|_{p,Q}\|w^{-\frac1{p}}\|_{\varphi,Q}
<\infty.
$$
Observe that in the case $A=I$, we recover the classical bump clases. For details of definitions, we refer to Section 2.
\begin{teo}\label{Fuerte}
Let $A$ be an invertible matrix, $1<p<\infty$. Suppose that $\varphi$ is a Young function such that $\overline{\varphi}\in B_{p}$ and  $w\in \mathcal{A}_{A,p,\varphi}$ then $M_{A^{-1}}$ is bounded on $L^p(w)$.
\end{teo}

\begin{remark}
Taking $\varphi(t)=t^{\frac{p}{p-\varepsilon - 1}}$ with $0<\varepsilon<1$, the class of weight is $\mathcal{A}_{A,p,\varphi}=\mathcal{A}_{A,p-\varepsilon}$ and $\overline{\varphi}\in B_p$ since $\overline{\varphi}(t)<t^{p}$. In  this case, the previous result said that if $w\in \mathcal{A}_{A,p-\varepsilon}$ then $M_{A^{-1}}$ is bounded on $L^p(w)$ and this implies that $w\in A_{\mathcal{A},p}$.
\end{remark}

When the class $\mathcal{A}_{A,p}$ is left open, the sufficient condition of Theorem \ref{Fuerte} yields a characterization of the weights for the boundedness of $M_{A^{-1}}$. A particular case where this occurs is the case of the matrices $A$ with finite order, see Theorem \ref{Afinite} below. The question of which matrices $A$ satisfy this condition remains open.

The result for the case of matrices with  finite order is,

\begin{teo}\label{Afinite}
Let $1<p<\infty$. Let $A$ be an invertible matrix such that $A^k =I$ for some $k\in \mathbb{N}$. Then
 $w\in  \mathcal{A}_{A,p}$ if and only if $M_{A^{-1}} $ is bounded on $L^p(w)$.

Moreover,  $$ \mathcal{A}_{A,p}=\{w \in \mathcal{A}_{p} : w(Ax)\lesssim w(x) \;a.e. \,x\in \mathbb{R}^n\}$$  and there exists a weight $w\in \mathcal{A}_{p} \setminus \mathcal{A}_{A,p}$.
\end{teo}

Now we present examples to illustrate the difference between our weight class $\mathcal{A}_{A,p}$ and the classical Muckenhoupt class $\mathcal{A}_{p}$.

Observe that there exist weights in  $\mathcal{A}_{A,p}$ that belong to the Muckenhoupt class, for example the radial weights or in particular the power weights. However, not all weights in the class $\mathcal{A}_{A,p}$ belong to a Muckenhoupt class, see Proposition \ref{Exam2}. We remark that Propositions 1.4 and 1.5 show that $\mathcal{A}_{A,p}$ and $A_p$ are independent classes, but do not resolve this question for general matrices.

First, we illustrate that not all Muckenhoupt weight satisfies $w(Ax)\lesssim w(x)$.
\begin{prop}\label{Exam1}
 Let $A=2I$. There exists a weight  $w\in \mathcal{A}_2$ such that $w\not \in \mathcal{A}_{A,2}$. Moreover, $w$ cannot satisfy the condition $w(Ax)\lesssim w(x)$.
\end{prop}

By contrast, in certain cases  the class $\mathcal{A}_{A,p}$ is not a subclass of the Muckenhoupt class. Indeed, we exhibit  a weight in $\mathcal{A}_{A,p}$ that does not belong to  $\mathcal{A}_{p}$. For this purpose, we consider $(\mathbb{R},\mu)$ with $\mu$ a radial non-doubling measure.

\begin{prop}\label{Exam2}
 Let $A=\lambda I$, with $0<\lambda<1$ and  $d\mu=e^{|x|}dx$ then there exists  $w\in A_{A,p}(\mu)$ such that $w\not \in A_{p}(\mu)$.
\end{prop}

Finally, we extend the results above to the fractional context. 
The fractional class of weight are defined as follows. A weight $w$  belongs in the class $ \mathcal{A}_{A,p,q}$, with  $1<p<\infty$, $\frac1{q}=\frac1{p}-\frac{\alpha}{n}$ and $0<\alpha<n$, if
$$
[w]_{\mathcal{A}_{A,p,q}}=\sup_Q \|w_A^{\frac1{q}}\|_{q,Q}\|w^{-\frac1{p}}\|_{p',Q}
<\infty.
$$

Let $\varphi$ be a Young function, a weight $w$  belongs in the class $ \mathcal{A}_{A,p,q,\varphi}$, $1<p,q<\infty$, if
$$
[w]_{\mathcal{A}_{A,p,q,\varphi}}=\sup_Q \|w_A^{\frac1{q}}\|_{q,Q}\|w^{-\frac1{p}}\|_{\varphi,Q}
<\infty.
$$

The fractional version of Theorem \ref{Fuerte} is 

\begin{teo}\label{FuerteFrac}
Let $A$ be an invertible matrix, $0<\alpha<n$, $1<p<\infty$ and $\frac1{q}=\frac1{p}-\frac{\alpha}{n}$. Suppose that $\varphi$ is a Young function such that $\overline{\varphi}\in B_{p}$ and  $w\in \mathcal{A}_{A,p,q,\varphi}$ then $M_{\alpha,A^{-1}}$ is bounded from $L^p(w^p)$ into $L^q(w^q)$.
\end{teo}

In the case of invertible matrices of finite order we have

\begin{teo}\label{FuerteFracAfinite}
Let $0<\alpha<n$, $1<p<\infty$ and $\frac1{q}=\frac1{p}-\frac{\alpha}{n}$. Let $A$ be an invertible matrix such that $A^k =I$ for some $k\in \mathbb{N}$. Then $ \mathcal{A}_{A,p,q}=\{w \in \mathcal{A}_{p,q} : w(Ax)\lesssim w(x) \;a.e. \,x\in \mathbb{R}^n\}$.

Moreover,  $w\in  \mathcal{A}_{A,p,q}$ if and only if $M_{\alpha,A^{-1}} $ is bounded from $L^p(w^p)$ into $L^{q}(w^q)$.
\end{teo}

As in the classic fractional case, one can verify that  $w\in \mathcal{A}_{A,p,p}$ if and only if  $w^p\in \mathcal{A}_{A,p}$. Then it is possible to adapt the previous examples to the fractional context.

The remainder of the paper is organized as follows. In Section 2, we provide some preliminaries and new properties of the class $\mathcal{A}_{A,p}$, Proposition 2.4 and Proposition 2.5. Section 3 is devoted
to the proofs of the main results, Theorem \ref{Fuerte} and  Theorem \ref{Afinite}. Later, in Section 4, we present the examples for Proposition \ref{Exam1} and Proposition \ref{Exam2}. Finally, in Section 5 we give extra commentaries of the fractional context.

\section{Preliminaries}

\subsection { Young Function and Luxemburg norm.}
Now, we present some extra definitions and properties for Young
functions. For more details of these
topics see \cite{O65} or \cite{RaoRen91}.

We recall  that a
function $\Psi : [0,\infty) \rightarrow [0,\infty)$ is said to be a
Young function if $\Psi$ is continuous, convex, no decreasing and
satisfies $\Psi(0)=0$ and $\displaystyle \lim_{t \rightarrow \infty}
\Psi(t)= \infty$.

For each Young function $\Psi$ we can induce an average of the
Luxemburg norm of a function $f$ in the cube $Q$ defined by
\begin{align*} \|f\|_{\Psi,Q}:= \inf \left\{\lambda >0:\, \frac1{|Q|} \int_{Q} \Psi\left(\frac{|f|}{\lambda}\right) \leq 1  \right\},\end{align*}
where $|Q|$ is the Lebesgue measure of $Q$. This function $\Psi$ has
an associated complementary Young function $\overline{\Psi}$
satisfying the generalized H\"older's inequality
 \begin{align*} \frac1{|Q|} \int_{Q} |fg| \leq 2\|f\|_{\Psi,Q}\|g\|_{\overline{\Psi},Q}. \end{align*}

The fractional maximal operator $M_{\alpha,\Psi}$  is defined in the
following way: given $f \in L^1_{\text{loc}}(\R)$ and $0\leq
\alpha<n$,
\begin{align*} M_{\alpha,\Psi}f(x) := {\underset{Q \ni x}{\sup}} |Q|^{\frac{\alpha}{n}}\|f\|_{\Psi,Q}. \end{align*}

Some examples of maximal operators related to certain Young
functions.
\begin{itemize}
\item $\Psi(t)=t$, then $\|f\|_{\Psi,Q}=f_Q:=\frac1{|Q|}\int_Q |f|$ and $M_{\alpha,\Psi}=M_{\alpha}$, the fractional maximal operator.

\item $\Psi(t)=t^{r}$ with $1< r<\infty$. In that case $\|f\|_{\Psi,Q}=\|f\|_{r,Q}:=\left(\frac1{|Q|}\int_Q |f|^r\right)^{1/r}$ and  
$M_{\alpha,\Psi}=M_{\alpha,r}$,  where $M_{0,r} f = M_rf  := M(f^r)^{1/r}$.

\item $\Psi(t)=\exp(t)-1$, then, $M_{\alpha,\Psi}=M_{\alpha,\exp(L)}$.

\item If $\beta > 0$ and $1\leq r<\infty$, $\Psi(t)= t^r\log(e+t)^{\beta}$ is a Young function then $M_{\alpha,\Psi}=M_{\alpha,L^r(\log L)^{\beta}}$.
\item If $\alpha=0$ and $k\in\mathbb{N}$, $\Psi(t)= t\log(e+t)^{k}$ it can be proved
that $M_{\Psi}\approx M^{k+1}$, where $M^{k+1}$ is Hardy-Littlewood
maximal operator, $M$, iterated $k+1$ times.
\end{itemize}

Pérez in \cite{P95} proved that the general maximal operator $M_{\Psi}$ is bounded on $L^p(\mathbb{R}^n)$ if and only if $\Psi \in B_p$. A Young function $\Psi \in B_p$ if $\int_{1}^{\infty} \frac{\Psi(t)}{t^p} \, \frac{dt}{t}<\infty$.

\subsection { Properties of  the class $\mathcal{A}_{A,p}$ }

Now, we present properties of maximal operator $M_{A^{-1}}$ and the class of weight  $\mathcal{A}_{A,p}$. For more details see \cite{IFR18,IFR20,IFRV22}, 

\begin{prop}\cite{IFR20}
Let $A$ be an invertible matrix. Then $M_{A^{-1}}w(x)=M(w_A)(x)$.
\end{prop}

\begin{prop}\cite{IFRV22}
Let $1\leq p<\infty$ and $A$ be an invertible matrix.
\begin{itemize}
\item If $w\in \mathcal{A}_p$ and $w(Ax)\lesssim w(x)$ a.e. $x\in \mathbb{R}^n$ then $w\in \mathcal{A}_{A,p}$.
\item If $w\in \mathcal{A}_{A,p}$ then $w(Ax)\leq [w]_{\mathcal{A}_{A,p}} w(x)$.
\end{itemize} 

\end{prop}
\begin{prop}\cite{IFRV22}
Let $1\leq p<\infty$, $A$ and $B$ be  invertible matrices. If $w\in \mathcal{A}_{A,p}$ and $w\in \mathcal{A}_{B,p}$, then $w\in \mathcal{A}_{AB,p}$.

Moreover, if $A^k=I$ for some $k\in\mathbb{N}$ and  $w\in \mathcal{A}_{A,p}$ then $w\in \mathcal{A}_{p}$.
\end{prop}

We now present new properties of the class $\mathcal{A}_{A,p}$.

\begin{prop}\label{Prop1}

Let $1< p<\infty$, $A$ be an invertible matrix and  $w\in \mathcal{A}_{A,p}$. If $S\subset Q$ then 
$$
\frac{|S|}{|Q|}\leq [w]_{\mathcal{A}_{A,p}}^{\frac1{p}}\left(\frac{w(S)}{w_A(Q)}\right)^{\frac1{p}}
$$
\end{prop}

\begin{proof}
Let $1<p<\infty$.  For $S\subset Q$, using Hölder's inequality we have
\begin{align*}
|S|
&=\int w^{\frac1{p}} \chi_S \chi^{-\frac1{p}} \chi_Q
\leq \left(\int_S w \right)^{\frac1{p}} \left(\int_Q w^{-\frac1{p-1}}\right)^{\frac1{p'}}
\\&= w(S)^{\frac1{p}} \left(\int_Q w^{-\frac1{p-1}}\right)^{\frac1{p'}} w_A(Q)^{\frac1{p}}w_A(Q)^{-\frac1{p}} \frac{|Q|}{|Q|}
\\&\leq \left(\frac{w(S)}{w_A(Q)}\right)^{\frac1{p}}[w]_{\mathcal{A}_{A,p}}^{\frac1{p}}|Q|.
\end{align*}
Then,
$$
\frac{|S|}{|Q|}\leq [w]_{\mathcal{A}_{A,p}}^{\frac1{p}}\left(\frac{w(S)}{w_A(Q)}\right)^{\frac1{p}}.
$$
\end{proof}

\begin{prop}
Let $1< p<\infty$ and $A,B,P$ be invertible matrices such that $A=PBP^{-1}$. Then $w\in \mathcal{A}_{A,p}$ if and only if $w_P\in \mathcal{A}_{B,p}$.
\end{prop}

\begin{proof}
Let  $A,B,P$ be invertible matrices such that $A=PBP^{-1}$. Observe that 
\[w_A(x)=w(Ax) =w(PBP^{-1}x)=w_P(BP^{-1})\quad \text{and} \quad w(x)=w_P(P^{-1}x)
\]

For each $Q$, we can take $\tilde{Q}=Q(P^{-1}c_Q, \|P^{-1}\|\ell_Q)$ such that $P^{-1}Q\subset \tilde{Q}$ and $|\tilde{Q}|=\|P^{-1}\|^n|Q|$.

Then, 
\begin{align*}
&\frac1{|Q|}\int_Q w_A (x) \,dx \left(\frac1{|Q|} \int_Q w^{-\frac1{p-1}}(x) \,dx\right)^{p-1}
\\&=\frac1{|Q|}\int_Q w_P (BP^{-1}x) \,dx \left(\frac1{|Q|} \int_Q w_P^{-\frac1{p-1}}(P^{-1}x) \,dx\right)^{p-1}
\\&=\frac{|\det(P)|}{|Q|}\int_{P^{-1}Q} w_P (By) \,dy \left(\frac{|\det(P)|}{|Q|} \int_{P^{-1}Q} w_P^{-\frac1{p-1}}(y) \,dy\right)^{p-1}
\\&=|\det(P)|^p\|P^{-1}\|^{np}\frac1{|\tilde{Q}|}\int_{P^{-1}Q} w_P (By) \,dy \left(\frac1{|\tilde{Q}|} \int_{P^{-1}Q} w_P^{-\frac1{p-1}}(y) \,dy\right)^{p-1}
\\&\leq c_P \frac1{|\tilde{Q}|}\int_{\tilde{Q}} w_P (By) \,dy \left(\frac1{|\tilde{Q}|} \int_{\tilde{Q}} w_P^{-\frac1{p-1}}(y) \,dy\right)^{p-1}
\end{align*}
and $[w]_{\mathcal{A}_{A,p}} \leq |\det(P)|^p\|P^{-1}\|^{np} [w_P]_{\mathcal{A}_{B,p}}$.

In a similar way, for each $Q$, we can take $\widehat{Q}=Q(P^{-1}c_Q, \|P^{-1}\|^{-1}\ell_Q)$ such that $\widehat{Q} \subset P^{-1}Q$ and $|\widehat{Q}|=\|P^{-1}\|^{-n}|Q|$. Moreover, 
\begin{align*}
&\frac1{|Q|}\int_Q w_A (x) \,dx \left(\frac1{|Q|} \int_Q w^{-\frac1{p-1}}(x) \,dx\right)^{p-1}
\\&\geq \frac{|\det(P)|^{p}}{\|P^{-1}\|^{np}} \frac1{|\widehat{Q}|}\int_{\widehat{Q}} w_P (By) \,dy \left(\frac1{|\widehat{Q}|} \int_{\widehat{Q}} w_P^{-\frac1{p-1}}(y) \,dy\right)^{p-1}
\end{align*}
Then, for $1<p<\infty$ we have $[w]_{\mathcal{A}_{A,p}} \geq |\det(P)|^p\|P^{-1}\|^{-np} [w_P]_{\mathcal{A}_{B,p}}$.

\end{proof}

\section{Proof of main results}

In this section we focus in the proof of the main results.  

\subsection{Proof of Theorem \ref{Fuerte}}

\begin{proof}[Proof of Theorem \ref{Fuerte}]
In the spirit of \cite{P95} we proceed as follows. 
Let $A$ be an invertible matrix, $1<p<\infty$ and $\varphi$ be a Young function such that $\overline{\varphi}\in B_{p}$.
For each integer $k$ and for any constant $a>2^n$, consider the following sets,
\begin{align*}
\Omega^{A}_k&=\left\{ x\in \R : M_{A^{-1}}f(x)>a^k  \right\},   &&\Omega_k=\left\{ x\in \R : Mf(x)>a^k  \right\}, 
\\
\mathcal{D}^A_k&=\left\{ x\in \R : M^d_{A^{-1}}f(x)>\frac{a^k}{4^n}  \right\}, && \mathcal{D}_k=\left\{ x\in \R : M^df(x)>\frac{a^k}{4}  \right\},
\end{align*}
where the maximal operators $M^d_{A^{-1}}$ and $M^d$ is the dyadic version of the previous operators.

Observe that $\Omega^{A}_k=A(\left\{ x\in \R : Mf(x)>a^k  \right\} )=A(\Omega_{k})$, $|\Omega^{A}_k|=|\det(A)||\Omega_k|$ and $\mathcal{D}^{A}_k=A(\mathcal{D}_{k})$, $|\mathcal{D}^{A}_k|=|\det(A)||\mathcal{D}_k|$.

Applying Calderón-Zygmund decomposition, there exists a family of dyadic cubes $\{Q_{k,j}\}_{k,j}$ for which $\Omega_{k}\subset \cup_j Q_{k,j}^3$ (where $Q^3$ is the cube with same center of $Q$ and side length $3l_Q$), $\mathcal{D}_{k}=\cup_j Q_{k,j}$ and 
$$ \frac{a^k}{4^n} < \frac1{|Q_{k,j}|}\int_{Q_{k,j}} |f(y)| \,dy \leq \frac{a^k}{2^n}   $$

Then, 
 $\Omega^A_{k}\subset \cup_j A(Q_{k,j}^3)$ and $\mathcal{D}^A_{k}=\cup_j A(Q_{k,j})$.

Now, we estimate 
\begin{align*}
\int_{\mathbb{R}^n} M_{A^{-1}}f(y)^p w(y) \,dy
&=\sum_{k\in \Z} \int_{\Omega^A_{k}\setminus \Omega^A_{k+1}} M_{A^{-1}}f(y)^p w(y) \,dy
\\& \leq \sum_{k\in \Z} a^{p(k+1)} w(\Omega^A_{k}) 
\\& \leq  c_{n,p}\sum_{k\in \Z}   \left(\frac1{|Q_{k,j}|}\int_{Q_{k,j}} |f(y)| \,dy\right)^p w_A(Q_{k,j}^3) 
\\& \leq  c_{n,p}\sum_{k\in \Z}   \left(\frac1{|Q_{k,j}^3|}\int_{Q_{k,j}^3} |f(y)|w^{\frac1{p}} w^{-\frac1{p}} \,dy\right)^p w_A(Q_{k,j}^3) 
\\& \leq  c_{n,p}\sum_{k\in \Z}   \left\|fw^{\frac1{p}}\right\|_{\overline{\varphi},Q_{k,j}^3}^p \left\|w^{-\frac1{p}}\right\|_{\varphi,Q_{k,j}^3}^p w_A(Q_{k,j}^3) 
\\& \leq  c_{n,p} [w]_{\mathcal{A}_{A,p,\varphi}}^p \sum_{k\in \Z}   \left\|fw^{\frac1{p}}\right\|_{\overline{\varphi},Q_{k,j}^3}^p |Q_{k,j}|.
\end{align*}

For each integer $k,j$ we let $E_{k,j}=Q_{k,j} \setminus Q_{k,j}\cap \mathcal{D}_{k+1}$. Then $\{E_{k,j}\}_{k,j}$ is a disjoint family of sets and there is a positive constant $\beta$ such that for each $k,j$, $|Q_{k,j}|<\beta |E_{k,j}|$. Then, since $\overline{\varphi}\in B_p$ we have
\begin{align*}
 \sum_{k\in \Z}   \left\|fw^{\frac1{p}}\right\|_{\overline{\varphi},Q_{k,j}^3}^p |Q_{k,j}|
 &\leq  \sum_{k\in \Z}   \left\|fw^{\frac1{p}}\right\|_{\overline{\varphi},Q_{k,j}^3}^p |E_{k,j}|
 \leq   \sum_{k\in \Z}   \int_{E_{k,j}} M_{\overline{\varphi}}(fw^{\frac1{p}})(x)^p \,dx
  \\& \leq    \int_{\mathbb{R}^n} M_{\overline{\varphi}}(fw^{\frac1{p}})(x)^p \,dx
\leq    \|fw^{\frac1{p}}\|_{p}^p
    =  \|f\|_{L^p (w)}^p.
\end{align*}

Hence, $M_{A^{-1}}$ is bounded on $L^p(w)$ and $\|M_{A^{-1}}f\|_{L^p(w)}\leq c_{n,p} [w]_{\mathcal{A}_{A,p,\varphi}} \|f\|_{L^p(w)}$.

\end{proof}

\subsection{Proof for matrix with finite order}
Now, we prove the results for matrix with finite order.

\begin{proof}[Proof of Theorem \ref{Afinite}]
Let $A$ be a matrix with finite order, that is $A^k=I$ for some $k\in \mathbb{N}$. 

Let $w\in \mathcal{A}_{A,p}$, this implies $w(Ax)\leq [w]_{\mathcal{A}_{A,p}}w(x)$ a.e. $x\in \mathbb{R}^n$ and $w\in \mathcal{A}_{A^j,p}$ for all $j\in \mathbb{N}$. In particular for $j=k$, that is $w\in \mathcal{A}_{A^k,p}= \mathcal{A}_{p}$. 
On the other hand, if $w\in \mathcal{A}_p$ such that $w(Ax)\lesssim w(x)$ a.e. $x\in \mathbb{R}$ then $w\in \mathcal{A}_{A,p}$.

Hence, $\mathcal{A}_{A,p} = \{w\in  \mathcal{A}_{p} : w(Ax) \lesssim w(x) \; a.e.\; x\in \mathbb{R}^n\}$.

To see the proper inclusion,  for $k\in \mathbb{N}$, we consider $w_k(x)=|x-k|^{-\frac1{2}}$ .
We define the weight
\[
w(x)= w_0(x) \chi_{\left(-\infty,\frac{1}{2}\right)}(x) + \sum_{k=1}^{\infty} w_k(x)\chi_{\left[k-\frac1{2},k +\frac1{2}\right)}(x).
\]
Observe that $w=M\left(\sum_{k\in \mathbb{N}_0 } \delta_{k}\right)^{\frac1{2}}$ and this implies $w\in A_2$. For instance we refer \cite[Theorem 3.4]{GCRFlibro}. 

For each $k\in \mathbb{N}$ we consider the intervals $J_k=(-k-\frac1{4},-k)$.
Note that if $x\in J_k$ then $w(-x)=w_k(-x)$, then 
\begin{align*}
\int_{J_k} w(-x)\,dx 
&=\int_{J_k}|x-k|^{-\frac1{2}}\,dx =1.
\end{align*}
On the other hand, if $x\in J_k\subset (-\infty,0)$ then $w(x)=|x|^{-\frac1{2}}$ and 
\begin{align*}
\int_{J_k} w(x)^{-1}\,dx 
=\int_{J_k}|x|^{\frac1{2}}\,dx 
\geq k^{\frac1{2}} |J_k|.
\end{align*}
Taking into account the estimates above we have
\[
\sup_{J} \frac1{|J|}\int_{J} w(-x)\,dx \frac1{|J|}\int_{J} w(x)^{-1}\,dx  
\geq  \sup_k\frac1{|J_k|}\int_{J_k} w(-x)\,dx \frac1{|T_k|}\int_{J_k} w(x)^{-1}\,dx 
=\infty
\]
and this implies that $w\not \in \mathcal{A}_{-I,2}$.

Finally, we now prove the characterization of the weights for the boundedness of $M_{A^{-1}}$

If $w\in \mathcal{A}_{A,p}=\{w \in \mathcal{A}_{p} : w(Ax)\lesssim w(x) \}$ then we have that $M_{A^{-1}}$ is bounded on $L^p(w)$. 
On the other hand, if $M_{A^{-1}}$ is bounded on $L^p(w)$ then $M_{A^{-1}}$ is bounded from $L^p(w)$ into $L^{p,\infty}$ and this implies $w\in \mathcal{A}_{A,p}$, see \cite[Theorem 2.3]{IFRV22}. 

Hence,  $M_{A^{-1}}$ is bounded on $L^p(w)$ if and only if $w\in \mathcal{A}_{A,p}$.

\end{proof}

\section{Examples}

Now, we present difference between the previous class and the classical Muckenhoupt weights.

\subsection{Not all Muckenhoupt weight is a weight in $\mathcal{A}_{A,p}$}

First we present a weight that  $w\in \mathcal{A}_2$ and $w\not \in \mathcal{A}_{A,2}$.

For $k\in \mathbb{N}$, let $c_k=2^{2k-1}$ and $a_k=3\cdot2^{2k-1}$. We consider $w_k(x)=|x-c_k|^{-\frac1{2}}$ in $\mathbb{R}$. We define the weight
\[
w(x)=w_1(x)\chi_{(-\infty,a_1)}(x) + \sum_{k=2}^{\infty} w_k(x)\chi_{(a_{k-1},a_k)}(x).
\]
Observe that $w=M\left(\sum_{k\in \mathbb{N}} \delta_{c_k}\right)^{\frac1{2}}$ and this implies $w\in A_2$. For instance we refer \cite[Theorem 3.4]{GCRFlibro}.

\begin{prop}
If $A=\begin{pmatrix} 2 \end{pmatrix}$, then the weight $w\not\in  \mathcal{A}_{A,2}$.
\end{prop}

\begin{proof}

For $k\in \mathbb{N}$, let $T_k=(2c_k,2c_k+\frac1{4})$.

First, we observe that if $a_0=0$ then $T_k\subset (a_{k-1},a_k)$ for all $k\in \mathbb{N}$

For $x\in T_k$, $w(x)=w_k(x)$ and 
\[
\int_{T_k} w(x)^{-1}\,dx
=\int_{T_k} w_k(x)^{-1}\,dx
=\int_{2c_k}^{2c_k+\frac1{4}}(2x-c_{k})^{\frac1{2}}\,dx
\geq |T_k| c_k ^{\frac1{2}}
\simeq  2^{k}.
\]

On the other hand, if $x\in T_k$ then $2x\in (a_k,a_{k+1})$, that implies $w(2x)=w_{k+1}(2x)$ and
\[
\int_{T_k} w(2x)\,dx
=\int_{T_k} w_{k+1}(2x)\,dx
=\int_{2c_k}^{2c_k+\frac1{4}}(2x-c_{k+1})^{-\frac1{2}}\,dx
=\int_{2c_k}^{2c_k+\frac1{4}}(2x-4c_k)^{-\frac1{2}}\,dx
=2^{-\frac1{2}}.
\]
Taking into account the estimates above and $|T_k|=\frac1{4}$ for all $k$,  we obtain
\[
\frac1{|T_k|}\int_{t_k} w(2x)\,dx \frac1{|T_k|}\int_{T_k} w(x)^{-1}\,dx \gtrsim 2^{k}. 
\]
Then, 
\[
\sup_{J} \frac1{|J|}\int_{J} w(2x)\,dx \frac1{|J|}\int_{J} w(x)^{-1}\,dx  
\geq  \sup_k\frac1{|T_k|}\int_{T_k} w(2x)\,dx \frac1{|T_k|}\int_{T_k} w(x)^{-1}\,dx 
=\infty
\]
and $w\not \in \mathcal{A}_{A,2}$.

\end{proof}

\subsection{Not all weight in $\mathcal{A}_{A,p}$ is a Muckenhoupt weight}

\

We consider the measure $d\mu=e^{|x|}dx$.
Note that $\mu$ is a non-doubling measure.

\begin{prop} Let  $1<p<\infty$, $w(x)=e^{(p-1)|x|}$ and let $A=\begin{pmatrix}\frac1{2}\end{pmatrix}$ be $1\times 1$ matrix. Then $w\in A_{A,p}(\mu)$ and $w\not \in A_{p}(\mu)$.
\end{prop}

\begin{proof}
Since the weight $w$ and the measure $\mu$ are radial, it is enough to prove it for intervals  in the positive half-line. 

Consider $R_h=(a,a+h)$ with $a\geq 0$ and $h>0$. Observe that 
$$\mu(R_h)=e^{a+h} (1-e^{-h}), \qquad  
\int_{R_h}w^{-\frac1{p-1}} \,d\mu=h$$ 
and 
\begin{align*}
\int_{R_h} w_A(x)\,d\mu(x)
&=\int_{a}^{a+h} e^{(p-1)|\frac1{2} x|}e^{|x|}\,dx
=\int_{a}^{a+h} e^{\left(p+1\right)\frac{x}{2}}\,dx
\\&=\frac{2}{p+1} e^{\left(p+1\right)\frac{a+h}{2}} \left[1-e^{-\left(p+1\right)\frac{h}{2}}\right].
\end{align*}
Then, 
\begin{align*}
&\frac1{\mu(R_h)}\int_{R_h} w_A(x)\,d\mu(x) \left(\frac1{\mu(R_h)}\int_{R_h} w(x)^{-\frac1{p-1}}\,d\mu(x)\right)^{p-1}
\\&= \frac{2}{p+1} e^{-\left(p-1\right)\frac{a+h}{2}}\frac{1-e^{-\left(p+1\right)\frac{h}{2}}}{1-e^{-h}}\left(\frac{h}{1-e^{-h}}\right)^{p-1}
\\&\leq  \frac{2}{p+1} \frac{1-e^{-\left(p+1\right)\frac{h}{2}}}{1-e^{-h}} \left(\frac{e^{-\frac{h}{2}}h}{1-e^{-h}}\right)^{p-1}
=:J_h
\end{align*}
where the last inequality holds since $a\geq 0$.

Since $J_h$ is non-negative, continuous on $(0,\infty)$ (in function of $h$), $\lim_{h\to 0} J_h = 1$  and $\lim_{h\to \infty} J_h =0$, 
we have that $\sup_{h>0}J_h<\infty$. Hence $[w]_{A_{A,p}(\mu)}<\infty$ and $w\in A_{A,p}(\mu)$.

Now, we proceed to prove $w\not \in A_{p}(\mu)$.  Observe that 
\begin{align*}
\int_{R_h} w(x)\,d\mu(x)
=\frac1{p} e^{pa} \left[e^{ph}-1\right].
\end{align*}
Then, 
\begin{align*}
&\frac1{\mu(R_h)}\int_{R_h} w(x)\,d\mu(x) \left(\frac1{\mu(R_h)}\int_{R_h} w(x)^{-\frac1{p-1}}\,d\mu(x)\right)^{p-1}
=\frac1{p}\frac{1-e^{-ph}}{(1-e^{-h})^p}  h ^{p-1}
\end{align*}
where the right side tends to infinity when $h$ tends to infinity. Then 
$w\not\in A_{p}(\mu)$.

\end{proof}

\section{Fractional maximal operator}

In this last section, we study the boundedness of fractional maximal operator $M_{\alpha,A^{-1}}$. We focus on the differences with the non-fractional setting.

We recall a weight $w$  belongs in the class $ \mathcal{A}_{A,p,q}$, $1<p,q<\infty$, if
$$
[w]_{\mathcal{A}_{A,p,q}}=\sup_Q \|w_A^{\frac1{q}}\|_{q,Q}\|w^{-\frac1{p}}\|_{p,Q}
<\infty.
$$

First we present properties of  $ \mathcal{A}_{A,p,q}$. For more details see \cite{IFR18,IFR20,IFRV22}, 

\begin{prop}
Let $1\leq p<\infty$, $1<q<\infty$ and $A$ be an invertible matrix.
\begin{itemize}
\item If $w\in \mathcal{A}_{p,q}$ and $w(Ax)\lesssim w(x)$ a.e. $x\in \mathbb{R}^n$ then $w\in \mathcal{A}_{A,p,q}$.
\item If $w\in \mathcal{A}_{A,p,q}$ then $w(Ax)\leq [w]_{\mathcal{A}_{A,p,q}} w(x)$.
\item Let $A,B$ be  invertible matrices. If $w\in \mathcal{A}_{A,p,q}$ and $w\in \mathcal{A}_{B,p,q}$, then $w\in \mathcal{A}_{AB,p,q}$.
\item If $A^k=I$ for some $k\in\mathbb{N}$ and  $w\in \mathcal{A}_{A,p,q}$ then $w\in \mathcal{A}_{p,q}$.
\end{itemize}
\end{prop}

Let $\varphi$ be a Young function, a weight $w$  belongs in the class $ \mathcal{A}_{A,p,q,\varphi}$, $1<p,q<\infty$, if
$$
[w]_{\mathcal{A}_{A,p,q,\varphi}}=\sup_Q \|w_A^{\frac1{q}}\|_{q,Q}\|w^{-\frac1{p}}\|_{\varphi,Q}
<\infty.
$$

\begin{proof}[Proof of Theorem \ref{FuerteFrac}]
We proceed as \cite{P94}. Let $A$ be an invertible matrix, $1<p<\infty$ and $\varphi$ be a Young function such that $\overline{\varphi}\in B_{p}$.
For each integer $k$ and for any constant $a>2^n$, consider the following sets,
\begin{align*}
\Omega^{A}_k&=\left\{ x\in \R : M_{\alpha,A^{-1}}f(x)>a^k  \right\}   &&\Omega_k=\left\{ x\in \R : M_{\alpha}f(x)>a^k  \right\} 
\\
\mathcal{D}^A_k&=\left\{ x\in \R : M^d_{\alpha,A^{-1}}f(x)>\frac{a^k}{4^n}  \right\} && \mathcal{D}_k=\left\{ x\in \R : M_{\alpha}^df(x)>\frac{a^k}{4}  \right\}
\end{align*}
Observe that $\Omega^{A}_k=A(\left\{ x\in \R : M_{\alpha}f(x)>a^k  \right\} )=A(\Omega_{k})$, $|\Omega^{A}_k|=|\det(A)||\Omega_k|$ and $\mathcal{D}^{A}_k=A(\mathcal{D}_{k})$, $|\mathcal{D}^{A}_k|=|\det(A)||\mathcal{D}_k|$.

Applying Calderón-Zygmund decomposition, there exists a family of dyadic cubes $\{Q_{k,j}\}_{k,j}$ for which $\Omega_{k}\subset \cup_j Q_{k,j}^3$ (where $Q^3$ is the cube with same center of $Q$ and side length $3l_Q$), $\mathcal{D}_{k}=\cup_j Q_{k,j}$ and 
$$ \frac{a^k}{4^n} < \frac1{|Q_{k,j}|^{1-\frac{\alpha}{n}}}\int_{Q_{k,j}} |f(y)| \,dy \leq \frac{a^k}{2^n}   $$

Then, 
 $\Omega^A_{k}\subset \cup_j A(Q_{k,j}^3)$ and $\mathcal{D}^A_{k}=\cup_j A(Q_{k,j})$ and

Now, we estimate 
\begin{align*}
\left[\int_{\mathbb{R}^n} M_{\alpha,A^{-1}}f(y)^q w^q(y) \,dy \right]^\frac1{q}
& \leq  c_{n,p,q}\left[ \sum_{k\in \Z}  |Q_{k,j}^3|^{\frac{\alpha}{n}q} \left\|fw^{\frac1{p}}\right\|_{\overline{\varphi},Q_{k,j}^3}^q \left\|w^{-\frac1{p}}\right\|_{\varphi,Q_{k,j}^3}^q w_A^q(Q_{k,j}^3) \right]^\frac1{q}
\\& \leq  c_{n,p,q} [w]_{\mathcal{A}_{A,p,q,\varphi}} \left[ \sum_{k\in \Z}   \left\|fw^{\frac1{p}}\right\|_{\overline{\varphi},Q_{k,j}^3}^q |Q_{k,j}|^\frac{q}{p}\right]^\frac1{q}
\\& \leq  c_{n,p,q} [w]_{\mathcal{A}_{A,p,q,\varphi}} \left[ \sum_{k\in \Z}   \left\|fw^{\frac1{p}}\right\|_{\overline{\varphi},Q_{k,j}^3}^p |Q_{k,j}|\right]^\frac1{p}
\end{align*}
The last inequality follows by Hölder’s inequality in $\ell^{q/p}$ with $p\leq q$ and using $\frac{q}{p} = 1 + \frac{\alpha q}{n}$.

For each integer $k,j$ we let $E_{k,j}=Q_{k,j} \setminus Q_{k,j}\cap \mathcal{D}_{k+1}$. Then $\{E_{k,j}\}_{k,j}$ is a disjoint family of sets and there is a positive constant $\beta$ such that for each $k,j$, $|Q_{k,j}|<\beta |E_{k,j}|$. Then, since $\overline{\varphi}\in B_p$ we have
\begin{align*}
\sum_{k\in \Z}   \left\|fw^{\frac1{p}}\right\|_{\overline{\varphi},Q_{k,j}^3}^p |Q_{k,j}|
 &\leq  \sum_{k\in \Z}   \left\|fw^{\frac1{p}}\right\|_{\overline{\varphi},Q_{k,j}^3}^p |E_{k,j}|
  \leq  \sum_{k\in \Z}   \int_{E_{k,j}} M_{\overline{\varphi}}(fw^{\frac1{p}})(x)^p \,dx
  \\& \leq    \int_{\mathbb{R}^n} M_{\overline{\varphi}}(fw^{\frac1{p}})(x)^p \,dx
     \leq    \|fw^{\frac1{p}}\|_{p}^p
    = \|f\|_{L^p (w)}^p.
\end{align*}

Hence, $M_{A^{-1}}$ is bounded from $L^p(w)$ into $L^q(w^q)$ and 
 $$\|M_{\alpha, A^{-1}}f\|_{L^q(w^q)}\leq c_{n,p,q} [w]_{\mathcal{A}_{A,p,q,\varphi}} \|f\|_{L^p(w)}.$$

\end{proof}

\section*{AI Disclosure}
Claude was used for language editing and proofreading. All mathematical results,
arguments, proofs, and conclusions were developed and proved by the author.

\end{document}